\documentclass[12pt]{article}
\usepackage{geometry,amsmath,amssymb,bbm,theorem,amscd}
\newcommand{\nequiv}{\not\equiv}
\newcommand{\nin}{\not\in}
\newcommand{\scdots}{\cdot\!\cdot\!\cdot}
\newcommand{\tmem}[1]{{\em #1\/}}
\newcommand{\tmmathbf}[1]{\ensuremath{\boldsymbol{#1}}}
\newcommand{\tmop}[1]{\ensuremath{\operatorname{#1}}}
\newcommand{\tmscript}[1]{\text{\scriptsize{$#1$}}}
\newcommand{\tmtextit}[1]{{\itshape{#1}}}
\newcommand{\tmtexttt}[1]{{\ttfamily{#1}}}
\newenvironment{proof}{\noindent\textbf{Proof\ }}{\hspace*{\fill}$\Box$\medskip}
\newtheorem{conjecture}{Conjecture}
{\theorembodyfont{\rmfamily}\newtheorem{example}{Example}}
\newtheorem{lemma}{Lemma}
\newtheorem{proposition}{Proposition}
\newtheorem{theorem}{Theorem}

\begin{document}


\title{Casselman's Basis of Iwahori vectors and the Bruhat
Order}\author{Daniel Bump and Maki Nakasuji}\maketitle

\begin{abstract}
  Casselman defined a basis $f_u$ of Iwahori fixed vectors of a spherical
  representation $(\pi, V)$ of a split semisimple $p$-adic group $G$ over a
  nonarchimedean local field $F$ by the condition that it be dual to the
  intertwining operators, indexed by elements $u$ of the Weyl group $W$. On
  the other hand there is a natural basis $\psi_u$ and one seeks to find the
  transition matrices between the two bases. Thus let $f_u = \sum_v \tilde{m}
  (u, v) \psi_v$ and $\psi_u = \sum_v m (u, v) f_v$. Using the Iwahori Hecke
  algebra we prove that if a combinatorial condition is satisfied then $m (u,
  v) = \prod_{\alpha} \frac{1 - q^{- 1} \tmmathbf{z}^{\alpha}}{1
  -\tmmathbf{z}^{\alpha}}$ where $\tmmathbf{z}$ are the Langlands parameters
  for the representation and $\alpha$ runs through the set $S (u, v)$ of
  positive coroots $\alpha \in \hat{\Phi}$ (the dual root system of $G$) such
  that $u \leqslant v r_{\alpha} < v$ with $r_{\alpha}$ the reflection
  corresponding to $\alpha$. The condition is conjecturally always satisfied
  if $G$ is simply-laced and the Kazhdan-Lusztig polynomial $P_{w_0 v, w_0 u}
  = 1$ with $w_0$ the long Weyl group element. There is a similar formula for
  $\tilde{m}$ conjecturally satisfied if $P_{u, v} = 1$. This leads to various
  combinatorial conjectures.
\end{abstract}

\section{Introduction}

Casselman~{\cite{CasselmanSpherical}} described an interesting basis of the
vectors in a spherical representation of a reductive $p$-adic group that are
fixed by the Iwahori subgroup. This basis is defined as being dual to the
standard intertwining operators. He remarked (p.402) that it was an unsolved
and apparently difficult problem to compute this basis explicitly. For his
applications, which include the computation of the spherical function and, in
Casselman and Shalika~{\cite{CasselmanShalika}} the spherical Whittaker
function, it is only necessary to compute one element of the basis explicitly.
Despite this difficulty, we began to look at the Casselman basis and we
obtained interesting partial results. These lead to some interesting
combinatorial questions about the Bruhat order.

Let $G$ be a split semisimple algebraic group over the nonarchimedean field
$F$. Let $B (F)$ be the standard Borel subgroup of $G (F)$, $K$ the standard
maximal compact subgroup, and $J$ the Iwahori subgroup of $K$. (See Section~2
for definitions of these.)

We write $B = T N$ where $T$ is the maximal split torus and $N$ its unipotent
radical. If $\chi$ is a character of $T (F)$ then $V (\chi)$ will be the
representation of $G (F)$ induced from $\chi$. Its space consists of locally
constant functions $f : G (F) \longrightarrow \mathbbm{C}$ such that
\[ f (b g) = (\delta^{1 / 2} \chi) (b g) \]
where $\delta : B (F) \longrightarrow \mathbbm{C}$ is the modular
quasicharacter and $\chi, \delta$ are extended to $B$ to be trivial on $N
(F)$. The action of $G (F)$ is by right translation.

If $\chi$ is in general position then $V (\chi)$ is irreducible. If $\chi$ is
unramified (which we assume) then the space $V (\chi)^J$ of $J$-fixed vectors
has dimension equal to the order of the Weyl group $W$, and so it is natural
to parametrize bases of $V (\chi)^J$ by $W$. There is one natural basis,
namely $\{\phi_w |w \in W\}$ defined as follows. If $b \in B (F), u \in W$ and
$k \in J$, define
\begin{equation}
  \label{standardiwahoribasis} \phi_w (b u k) = \left\{ \begin{array}{ll}
    \delta^{1 / 2} \chi (b) & \text{if $k \in J$, $u = w$,}\\
    0 & \text{otherwise} .
  \end{array} \right.
\end{equation}
It is clear that this is a basis of $V (\chi)^J$.

If $w \in W$ then there is an intertwining integral $M_w : V (\chi)
\longrightarrow V ( {^w \chi})$. It is given by (\ref{intertwiningdef}) below.
These have the property that if $l (w w') = l (w) + l (w')$ then $M_{w w'} =
M_w \circ M_{w'}$, where $l : W \longrightarrow \mathbbm{Z}$ is the length
function. The Casselman basis $\{f_w |w \in W\}$ is the basis defined by the
condition that
\[ (M_w f_v) (1) = \left\{ \begin{array}{ll}
     1 & \text{if $w = v$,}\\
     0 & \text{if $w \neq v$.}
   \end{array} \right. \]
The question of Casselman mentioned above is to express the basis $f_w$ in
terms of the basis $\phi_w$. However we found it better to try to express it
in terms of the basis
\[ \psi_u = \sum_{v \geqslant u} \phi_v, \]
where $\geqslant$ is the Bruhat order. By Verma~{\cite{Verma}} or
Stembridge~{\cite{Stembridge}}
\[ \phi_u = \sum_{v \geqslant u} (- 1)^{l (v) - l (u)} \psi_v, \]
so expressing the $f_w$ in terms of $\psi_w$ is equivalent to Casselman's
question.

This problem can be divided into two parts: first, to compute the values of $m
(u, v) = (M_v \psi_u) (1)$, and second, to invert the matrix $m (u, v)_{u, v
\in W}$. Indeed, if $\tilde{m} (u, v)_{u, v \in W}$ is the inverse matrix, so
$\sum_v \tilde{m} (u, v) m (v, w) = \delta_{u, w}$ (Kronecker $\delta$) then
$\sum_u \tilde{m} (v, u) \psi_u$ will satisfy $M_w \left( \sum_u \tilde{m} (v,
u) \psi_u \right) (1) = \delta_{v, w}$ and so $f_v = \sum_u \tilde{m} (v, u)
\psi_u$ is the Casselman basis.

Let $\hat{\Phi}$ be the root system with respect to $\hat{T}$, the dual torus
of $T$. This is a complex torus in the L-group $^L G$.

If $\alpha \in \hat{\Phi}^+$ let $r_{\alpha}$ be the reflection in the
hyperplace perpendicular to $\alpha$. Thus if $\alpha$ is simple, $r_{\alpha}$
is the simple reflection $s_{\alpha}$. Then for any $u \leqslant y \leqslant
v$ we have
\[ \#\{\alpha \in \hat{\Phi}^+ | u \leqslant y.r_{\alpha} \leqslant v\}
   \geqslant l (v) - l (u) . \]
This statement is known as {\tmem{Deodhar's conjecture}}. The condition is
sometimes written $u \leqslant r_{\alpha} .y \leqslant v$ but this does not
change the cardinality of the set since $y.r_{\alpha} = r_{\beta} .y$ for
another positive root $\beta = \pm y (\alpha)$. This inequality was stated by
Deodhar~{\cite{DeodharConjecture}} who proved it in some cases; the general
statement is a theorem of Dyer~{\cite{Dyer}} and (independently)
Polo~{\cite{Polo}} and Carrell and Peterson (Carrell {\cite{Carrell}}). In
particular, taking $y = v$ or $u$ gives
\[ S (u, v) =\{\alpha \in \hat{\Phi}^+ |u \leqslant v r_{\alpha} < v\},
   \hspace{2em} S' (u, v) =\{\alpha \in \hat{\Phi}^+ |u \leqslant u r_{\alpha}
   < v\}. \]
Then Deodhar's conjecture implies that $S (u, v)$ and $S' (u, v)$ have
cardinality $\geqslant l (v) - l (u)$.

\begin{proposition}
  If the Kazhdan-Lusztig polynomial $P_{u, v} = 1$ then $|S' (u, v) | = l (v)
  - l (u)$. If the Kazhdan-Lusztig polynomial $P_{w_0 v, w_0 u} = 1$ then $|S
  (u, v) | = l (v) - l (u)$.
\end{proposition}

\begin{proof}
  The first statement follows from Carrell~{\cite{Carrell}}, Theorem~C. If
  $P_{w_0 v, w_0 u} = 1$ then it follows that $|S' (w_0 v, w_0 u) | = l (w_0
  u) - l (w_0 v)$. Since $x \leqslant y$ if and only if $w_0 y \leqslant w_0
  x$, this is equivalent to $|S (u, v) | = l (v) - l (u)$.
\end{proof}

We assume that $\hat{\Phi}$ is simply-laced, that is, of Cartan type $A$, $D$
or $E$. In this case, we make the following conjectures. The unramified
character $\chi = \chi_{\tmmathbf{z}}$ of $T (F)$ is parametrized by an
element $\tmmathbf{z}$ of the complex torus $\hat{T}$ in the L-group $^L G$.
(See Section~\ref{preliminaries}.)

\begin{conjecture}
  \label{conjecture1}Assume that $\hat{\Phi}$ is simply-laced. Suppose that $u
  \leqslant v$ in the Bruhat order. In this case $w_0 v \leqslant w_0 u$.
  Suppose that the $|S (u, v) | = l (v) - l (u)$. Then we conjecture that
  \begin{equation}
    \label{meval} (M_v \psi_u) (1) = \prod_{\alpha \in S (u, v)} \frac{1 -
    q^{- 1} \tmmathbf{z}^{\alpha}}{1 -\tmmathbf{z}^{\alpha}} .
  \end{equation}
\end{conjecture}

\begin{conjecture}
  \label{conjecture2}Assume that $\hat{\Phi}$ is simply-laced. Suppose that $u
  \leqslant v$ in the Bruhat order. Suppose that $|S' (u, v) | = l (v) - l
  (u)$. Then we conjecture that
  \begin{equation}
    \label{mtildeval} \tilde{m} (u, v) = (- 1)^{|S' (u, v) |} \prod_{\alpha
    \in S' (u, v)} \frac{1 - q^{- 1} \tmmathbf{z}^{\alpha}}{1
    -\tmmathbf{z}^{\alpha}} .
  \end{equation}
\end{conjecture}

We give an example to show that the assumption that $\hat{\Phi}$ is
simply-laced is necessary. Let $\hat{\Phi}$ have Cartan type $B_2$, with
$\alpha_1, \alpha_2$ being the long and short simple roots, respectively and
$\sigma_1 = s_{\alpha_1}, \sigma_2 = s_{\alpha_2}$ being the simple
reflections. Then we find that when $(u, v) = (\sigma_1, \sigma_1 \sigma_2
\sigma_1)$ or $(\sigma_1, \sigma_1 \sigma_2 \sigma_1 \sigma_2)$ the conclusion
of Conjecture~\ref{conjecture1} fails, though the Kazhdan-Lusztig polynomial
$P_{w_0 v, w_0 u} = 1$. Nevertheless the conjecture is {\tmem{often}} true for
type $B_2$, for these are the only failures. There are 33 pairs $(u, v)$ with
$u \leqslant v$, and Conjecture~\ref{conjecture1} gives the correct value for
$(M_v \psi_u) (1)$ in every case except for these two. Hence it becomes
interesting to ask how the hypothesis in Conjectures~\ref{conjecture1}
and~\ref{conjecture2} should be modified if when $\hat{\Phi}$ is not
simply-laced.

We recall the formula of Gindikin and Karpelevich. Let $\phi^{\circ} =
{^{\chi} \phi^{\circ}}$ be the standard spherical vector in $\tmop{Ind}_B^G
(\delta^{1 / 2} \chi)$ defined by $\phi^{\circ} (b k) = \delta^{1 / 2} \chi
(b)$ when $b \in B (F)$ and $k \in K$. In this case
\begin{equation}
  \label{gindikinkarpelevich} M (v) {^{\chi} \phi^{\circ}} = \left[
  \prod_{\tmscript{\begin{array}{c}
    \alpha \in \hat{\Phi}^+\\
    v (\alpha) \in \hat{\Phi}^-
  \end{array}}} \frac{1 - q^{- 1} \tmmathbf{z}^{\alpha}}{1
  -\tmmathbf{z}^{\alpha}} \right] {^{^v \chi} \phi^{\circ}} .
\end{equation}
This well-known formula was proved by Langlands~{\cite{LanglandsEuler}} after
Gindikin and Karpelevich proved a similar statement for real groups. See
Theorem~3.1 of Casselman~{\cite{CasselmanSpherical}} for a proof.

\begin{theorem}
  If $u = 1$ then Conjecture~1 is true.
\end{theorem}

\begin{proof}
  We will deduce this from (\ref{gindikinkarpelevich}). In this case $\psi_1 =
  \phi^{\circ}$, so to prove Conjecture~1 we need to know that if $\alpha \in
  \hat{\Phi}^+$ then $\alpha \in S (1, v)$ if and only if $v (\alpha) \in
  \hat{\Phi}^-$. This follows from Proposition~\ref{wralphacrit} with $w = v$.
\end{proof}

Thus Conjecture~\ref{conjecture1} generalizes the formula of Gindikin and
Karpelevich. If $u \neq 1$ it resembles the formula of Gindikin and
Karpelevich but there are some important differences, which we will now
discuss.

We say that a subset $S$ of $\hat{\Phi}$ is {\tmem{convex}} if $\alpha \in S$
implies $- \alpha \nin S$ and whenever $\alpha, \beta \in S$ and $\alpha +
\beta \in \hat{\Phi}$ we have $\alpha + \beta \in S$. The set $S (1, v)
=\{\alpha \in \hat{\Phi}^+ | v (\alpha) \in \hat{\Phi}^- \}$ is convex in this
sense. Moreover it has the property that if it is nonempty then it contains
simple roots; this follows from the fact that its complement in $\hat{\Phi}^+$
is $\{\alpha \in \hat{\Phi}^+ | v (\alpha) \in \hat{\Phi}^+ \}$, which is also
convex. These are special properties that $S (u, v)$ may not have in general.

\begin{example}
  Suppose that $\hat{\Phi} = A_2$ with simple roots $\alpha_1$ and $\alpha_2$
  and simple reflections $\sigma_i = s_{\alpha_i}$. Let $u = \sigma_1$, $v =
  w_0 = \sigma_1 \sigma_2 \sigma_1$. Then $S (u, v) =\{\alpha_1, \alpha_2 \}$
  is not convex.
\end{example}

\begin{example}
  Suppose that $\hat{\Phi} = A_2$ and that $u = \sigma_2$, $v = \sigma_1
  \sigma_2$. Then $S (u, v) =\{\alpha_1 + \alpha_2 \}$. Thus $S (u, v)$
  contains no simple roots.
\end{example}

We see that $S (u, v)$ has two special properties in the case where $u = 1$,
namely that it is convex and that its complement is convex, which implies that
(if nonempty) it always contains simple roots. These properties fail for
general $u$.

We turn now to an interesting combinatorial conjecture which implies
Conjecture~\ref{conjecture1}.

Let $W$ be a Coxeter group with generators $\Sigma$, whose elements will be
referred to as {\tmem{simple reflections}}. If $u, v \in W$ and $u \leqslant
v$ with respect to the Bruhat order, then we will define the notion of a
{\tmem{good word}} for $v$ with respect to $u$. First, this is a reduced
decomposition $v = s_1 \cdots s_n$ into a product of simple reflections, where
$n$ equals the length $l (v)$. It has the following property. Let $S$ be the
set of integers $j$ such that
\[ u \leqslant s_1 \cdots \widehat{s_j} \cdots s_n, \]
where the ``hat'' means that the factor $s_j$ is omitted. Let $S =\{j_1,
\cdots, j_d \}$, which we arrange in ascending order: $j_1 < \cdots < j_d$.
Then we say that the decomposition $s_1 \cdots s_n$ is a {\tmem{good word}}
for $v$ with respect to $u$ if
\[ u = s_1 \cdots \widehat{s_{j_1}} \cdots \widehat{s_{j_d}} \cdots s_n . \]
Now $d$ has an intrinsic characterization in terms of $u$ and $v$ independent
of the decomposition $v = s_1 \cdots s_n$. It is the number of reflections $r$
in $W$ such that $u \leqslant v r < v$. Indeed, given any reflection $r$ such
that $u \leqslant v r < v$ there is a unique $j$ such that
\[ r = s_j s_{j + 1} \cdots s_{n - 1} s_n s_{n - 1} \cdots s_j \]
and so $v r = s_1 \cdots \widehat{s_j} \cdots s_n$. Thus $d = |S (u, v) |$ and
by Deodhar's conjecture $d \geqslant l (v) - l (u)$. Therefore a good word can
exist only if $d = l (v) - l (u)$.

Let us consider some examples. First consider the case where $W = A_2$, with
generators $\sigma_1 = s_{\alpha_1}$ and $\sigma_2 = s_{\alpha_2}$ satisfying
$\sigma_i^2 = 1$ and $(\sigma_1 \sigma_2)^3 = 1$. Let $u = \sigma_1$ and $v =
\sigma_1 \sigma_2 \sigma_1$. Then $\sigma_1 \sigma_2 \sigma_1$ is not a good
word for $v$ with respect to $u$, since
\[ \sigma_1 \leqslant \widehat{\sigma_1} \sigma_2 \sigma_1, \hspace{2em}
   \sigma_1 \leqslant \sigma_1 \sigma_2 \widehat{\sigma_1}, \hspace{2em}
   \text{but $\sigma_1 \neq \widehat{\sigma_1} \sigma_2 \widehat{\sigma_1} .$}
\]
But $v = \sigma_2 \sigma_1 \sigma_2$ by the braid relation, and this word is
good. Indeed, we have
\[ \sigma_1 \leqslant \widehat{\sigma_2} \sigma_1 \sigma_2, \hspace{2em}
   \sigma_1 \leqslant \sigma_2 \sigma_1 \widehat{\sigma_2}, \hspace{2em}
   \sigma_1 = \widehat{\sigma_2} \sigma_1 \widehat{\sigma_2} . \]
\begin{conjecture}
  \label{conjecturegood}If $W$ is simply-laced and $d = l (v) - l (u)$ then
  $v$ admits a good word with respect to $u$.
\end{conjecture}

\begin{proposition}
  \label{computerreport}Conjecture~\ref{conjecturegood} is true for $W = A_4$
  or $D_4$.
\end{proposition}

\begin{proof}
  This was established by computer computation using {\sc Sage}.
\end{proof}

If $W$ is not simply-laced, then this fails: for example, let $W = B_2$, with
generators $\sigma_1$ and $\sigma_2$ satisfying $\sigma_i^2 = 1$ and
$(\sigma_1 \sigma_2)^4 = 1$. Let $u, v = \sigma_1, \sigma_1 \sigma_2
\sigma_1$. Then there is no good word for $v$ with respect to $u$. It is an
interesting question to give other characterizations (for example in terms of
Schubert varieties) of the pairs $u, v$ such that $v$ admits a good word for
$u$ when $W$ is not simply-laced.

Our main theorem is the following result:

\begin{theorem}
  \label{maintheorem}If $v$ admits a good word for $u$ then (\ref{meval}) is
  true.
\end{theorem}

By Theorem~\ref{maintheorem}, Conjecture~\ref{conjecturegood} implies
Conjecture~\ref{conjecture1}. Theorem~\ref{maintheorem} is true whether or not
$\hat{\Phi}$ is simply-laced. However, as we have mentioned, if $\hat{\Phi}$
is not simply-laced, there may not exist a good word even if $l (v) - l (u) =
d$.

By Proposition~\ref{computerreport} it follows that
Conjecture~\ref{conjecture1} is true if $G = \tmop{GL}_r$ with $r \leqslant 5$
or $G = \tmop{SO} (8)$ (split).

We have investigated Conjecture~\ref{conjecture2} less than
Conjecture~\ref{conjecture1} and have less evidence for it.
Conjecture~\ref{conjecture2} also is related to a combinatorial conjecture
which we will now state.

\begin{conjecture}
  \label{conjecture3}Assume that $\hat{\Phi}$ is simply-laced. If $u < v$ and
  $P_{u, v} = 1$ then there exists $\beta \in \hat{\Phi}^+$ such that $u
  \leqslant t \leqslant v$ if and only if $u \leqslant r_{\beta} t \leqslant
  v$.
\end{conjecture}

It is shown in Proposition~3.7 of Deodhar~{\cite{DeodharCharacterizations}}
that the Bruhat interval $[u, v] =\{t|u \leqslant t \leqslant v\}$ has as many
elements of even length as of odd length. Conjecture~\ref{conjecture3} (when
applicable) gives a strengthening of this since $t \mapsto r_{\beta} t$ is a
specific bijection of $[u, v]$ to itself that interchanges elements of odd and
even length.

We have checked using a computer that Conjecture~\ref{conjecture3} is true for
$A_r$ when $r \leqslant 4$. For example if $\hat{\Phi} = A_3$ then there
exists such a $\beta$ for every pair $u \leqslant v$ except the pair
$\sigma_2, \sigma_2 \sigma_1 \sigma_3 \sigma_2$ and $\sigma_1 \sigma_3$,
$\sigma_1 \sigma_3 \sigma_2 \sigma_1 \sigma_3$. For these pairs, we have $u
\prec v$ (in the notation of Kazhdan and Lusztig~{\cite{KazhdanLusztig}}) but
$l (v) > l (u) + 1$ and so $P_{u, v} \neq 1$. For $A_4$, there are pairs $u
\leqslant v$ such that $u \prec v$ is not true but still the Bruhat interval
$\{t|u \leqslant t \leqslant v\}$ is not stabilized for any simple reflection.
However for these examples we have $P_{u, v} \neq 1$ and $P_{w_0 v, w_0 u}
\neq 1$, and Conjecture~\ref{conjecture3} is still true.

We will prove in Theorem~\ref{conjtwoprogress} that
Conjecture~\ref{conjecture3} and Conjecture~\ref{conjecture1} together imply a
weak form of Conjecture~\ref{conjecture2}.

When this work was at an early stage we spoke with Thomas Lam, Anne Schilling,
Mark Shimozono, Nicolas Thi\'ery and others and their remarks were helpful in
correctly formulating Conjecture~\ref{conjecture1}. We also thank Ben
Brubaker, Gautam Chinta, Solomon Friedberg and Paul Gunnells for helpful
conversations.

{\sc Sage} mathematical software~{\cite{SAGE}} was crucial in these investigations.
(Versions 4.3.2 and later have support for Iwahori Hecke algebras and Bruhat
order.)

This work was supported in part by the JSPS Research Fellowship for Young
Scientists and by NSF grant DMS-0652817.

\section{Preliminaries\label{preliminaries}}

Let $\mathfrak{g}_{\mathbbm{C}}$ be a semisimple Lie algebra over
$\mathbbm{C}$. Let $\mathfrak{t}_{\mathbbm{C}}$ be a split Cartan subalgebra
of $\mathfrak{g}$. Let $\Phi$ be the root system of
$\mathfrak{g}_{\mathbbm{C}}$ corresponding to $\mathfrak{t}$ and let $W$ be
the Weyl group, and let $\hat{\Phi}$ be the dual root system.

Let $H_{\alpha} \in \mathfrak{t}$ ($\alpha \in \Phi$) be the coroots. Thus the
root $\alpha$ is the linear functional $x \longmapsto \frac{2 \left\langle x,
H_a \right\rangle}{\left\langle H_{\alpha}, H_{\alpha} \right\rangle}$ with
respect to a fixed $W$-invariant inner product on $\mathfrak{t}$. Using
Th\'eor\`eme~1 of Chevalley~{\cite{Chevalley}} we may choose a basis
$\mathfrak{g}$ that consists of $X_{\alpha}, X_{- \alpha}$ where $\alpha$ runs
through the set $\Phi^+$ of positive roots and $H_{\alpha} \in \mathfrak{t}$
where $\alpha$ runs through the simple roots. These have the properties that
$[X_{\alpha}, X_{\beta}] = \pm (p + 1) X_{\alpha + \beta}$ when $\alpha,
\beta, \alpha + \beta \in \Phi$ is a root, where $p$ is the greatest integer
such that $\beta - p \alpha \in \Phi$ and $[H_{\alpha}, X_{\beta}] = \frac{2
\left\langle \alpha, \beta \right\rangle}{\left\langle \alpha, \alpha
\right\rangle} X_{\beta}$. Let $\mathfrak{g}_{\mathbbm{Z}}$ be the lattice
spanned by this Chevalley basis. It is a Lie algebra over $\mathbbm{Z}$ such
that $\mathfrak{g}_{\mathbbm{C}} =\mathbbm{C} \otimes
\mathfrak{g}_{\mathbbm{Z}}$.

Now if $F$ is a field let $\mathfrak{g}_F = F \otimes
\mathfrak{g}_{\mathbbm{Z}}$. We will take $F$ to be a nonarchimedean local
field. Let $G$ be a split semisimple algebraic group defined over $F$ with Lie
algebra $\mathfrak{g}_F$. Let $\mathfrak{o}$ be the ring of integers in $F$,
$\mathfrak{p}$ the maximal ideal of $\mathfrak{o}$ and $q$ the cardinality of
the residue field.

If $\alpha \in \Phi^+$ then there exists a homomorphism $i_{\alpha} :
\tmop{SL}_2 \longrightarrow G$ such that under the differential $d i_{\alpha}
: \mathfrak{s}\mathfrak{l}_2 \longrightarrow \mathfrak{g}$ we have
\[ d i_{\alpha} \left(\begin{array}{cc}
     0 & 1\\
     0 & 0
   \end{array}\right) = X_{\alpha}, \hspace{2em} d i_{\alpha}
   \left(\begin{array}{cc}
     0 & 0\\
     1 & 0
   \end{array}\right) = X_{- \alpha}, \hspace{2em} d i_{\alpha}
   \left(\begin{array}{cc}
     1 & 0\\
     0 & - 1
   \end{array}\right) = H_{\alpha} . \]
Let $x_{\alpha} : F \longrightarrow G (F)$ be the one-parameter subgroup
$x_{\alpha} (t) = \exp (t X_{\alpha})$. The Borel subgroup $B (F) = N (F) T
(F)$ where $T (F)$ is the split Cartan subgroup with $\tmop{Lie} (T)
=\mathfrak{t}$ and $N$ is generated by the $x_{\alpha} (F)$ with $\alpha \in
\Phi^+$. If $\mathfrak{a}$ is a fractional ideal we will also denote by $N
(\mathfrak{a})$ the subgroup generated by $x_{\alpha} (\mathfrak{a})$ with
$\alpha \in \Phi^+$. Similarly $N_- (F)$ and $N_- (\mathfrak{a})$ are
generated by $x_{- \alpha} (F)$ or $x_{- \alpha} (\mathfrak{a})$ with $\alpha
\in \Phi^+$, and $B_- (F) = N_- (F) T (F)$. Let $w_0$ be the long element of
$W$. Let $a_{\alpha} = i_{\alpha} \left(\begin{array}{cc}
  p & \\
  & p^{- 1}
\end{array}\right)$ where $p$ is a fixed generator of $\mathfrak{p}$.

Let $K$ be the maximal compact subgroup of $G (F)$ that stabilizes
$\mathfrak{g}_{\mathfrak{o}}$ in the adjoint representation. Then reduction
modulo $\mathfrak{p}$ gives a homomorphism $K \longrightarrow G
(\mathbbm{F}_q)$. Let $J$ be the preimage of $B (\mathbbm{F}_q)$ under this
homomorphism. This is the {\tmem{Iwahori subgroup}}.

By a result of Iwahori and Matsumoto~{\cite{IwahoriMatsumoto}} (Section 2), we
have a generalized Tits system in $G (F)$ with respect to $J$ and the
normalizer $N$ of the maximal torus $T$ of $G$ that has Lie algebra
$\mathfrak{t}_F = F \otimes \mathfrak{t}$. See also Iwahori~{\cite{Iwahori}}.
The subgroup denoted $B$ in these papers and in Matsumoto~{\cite{Matsumoto}}
is actually $w_0 J w_0^{- 1}$. This is a bornological $(B, N)$-pair in the
sense of Matsumoto~{\cite{Matsumoto}}, and we may make use of his results. In
particular we have the Iwasawa decomposition $G (F) = B (F) K$ and let $T
(\mathfrak{o}) = T (F) \cap K$. The Iwahori subgroup $J$ is the subgroup
generated by $T (\mathfrak{o})$, $N (\mathfrak{o})$ and $N_- (\mathfrak{p})$.

We have the {\tmem{Iwahori factorization}}, which is the statement that the
multiplication map $T (\mathfrak{o}) \times N_- (\mathfrak{p}) \times N
(\mathfrak{o}) \longrightarrow J$ is a homeomorphism. The three factors for
this may be taken in any order. See Matsumoto~{\cite{Matsumoto}} Proposition
5.3.3.

Let $\chi$ be a quasicharacter of $T (F)$. We say $\chi$ is
{\tmem{unramified}} if $\chi$ is trivial on $T (\mathfrak{o})$. \ Let
$X^{\ast} (T (F) / T (\mathfrak{o}))$ be the group of unramified
quasicharacters. It is isomorphic to $X^{\ast} (\mathbbm{Z}^r) =\mathbbm{C}^r$
where $r$ is the rank of $G$. The (connected) L-group $\hat{G} = {^L
G^{\circ}}$ defined by Langlands~{\cite{LanglandsEuler}} is a complex analytic
group with a maximal torus $\hat{T}$ such that the unramified quasicharacters
of $T (F)$ are in bijection with the elements of $\hat{T}$. If $\tmmathbf{z}
\in \hat{T}$ let $\chi_{\tmmathbf{z}}$ be the corresponding unramified
quasicharacter.

The Weyl groups $N_G (T) / T$ and $N_{\hat{G}} ( \hat{T}) / \hat{T}$ are
isomorphic and may be identified. If $\tmmathbf{z} \in \hat{T}$ and $w \in W$
then $\chi_{w (\tmmathbf{z})} =^w \chi_{\tmmathbf{z}}$ where $^w \chi (t) =
\chi (w^{- 1} t w)$. If $\chi = \chi_{\tmmathbf{z}}$ is an unramified
quasicharacter let $V (\chi) = \tmop{Ind}_B^G (\delta^{1 / 2} \chi)$ denote
the space of locally constant functions $f$ on $G (F)$ such that if $b \in B
(F)$ then
\[ f (b g) = (\delta^{1 / 2} \chi) (b) \, f (g) \]
where $\delta : B (F) \longrightarrow \mathbbm{C}$ is the modular
quasicharacter. This is a module for $G (F)$ under right translation, and if
$\tmmathbf{z}$ is in general position it is irreducible. The standard
intertwining operators $M_w : V (\chi) \longrightarrow V ( {^w \chi})$ are
defined by
\begin{equation}
  \label{intertwiningdef} (M_w f) (g) = \int_{N \cap w N_- w^{- 1}} f (w^{- 1}
  n g) \, d n = \int_{(N \cap w N w^{- 1}) \backslash N} f (w^{- 1} n g) \, d
  n.
\end{equation}
The integral is absolutely convergent if $| \chi (a_{\alpha}) | < 1$, and may
be meromorphically continued to all $\chi$.

We recall that $\phi_w$ defined by (\ref{standardiwahoribasis}) are a basis of
$V (\chi)^J$. By the Iwasawa decomposition, $G (F) = B (F) K$ and by the
Bruhat decomposition for $G (\mathbbm{F}_q)$ pulled back to $K$ under the
canonical map we have $K = \bigcup_{u \in W} J u J = \bigcup_{u \in W} B
(\mathfrak{o}) u J$. Therefore
\[ G (F) = \bigcup_{u \in W} B (F) u J \hspace{2em} \text{(disjoint)} . \]
\begin{proposition}
  \label{rootlist}Let $x \in W$ and let $w = s_{i_1} \cdots s_{i_k}$ be a
  reduced decomposition into simple reflections. Then
  \begin{equation}
    \label{rootlista} \left\{ \alpha \in \hat{\Phi}^+ | w (\alpha) \in
    \hat{\Phi}^- \}=\{\alpha_{i_k}, s_{i_k} (\alpha_{i_{k - 1}}), s_{i_k}
    s_{i_{k - 1}} (\alpha_{i_{k - 2}}), \cdots, s_{i_k} \cdots s_{i_2}
    (\alpha_{i_1})\}. \right.
  \end{equation}
  The elements in this list are distinct, so $k = l (x)$ is the cardinality of
  this set.
\end{proposition}

\begin{proof}
  This is Corollary 2 to Proposition~17 in VI.1.6 of
  Bourbaki~{\cite{Bourbaki}}.
\end{proof}

\begin{proposition}
  \label{wralphacrit}Let $w \in W$. If $w (\alpha) \in \hat{\Phi}^-$ then $w
  r_{\alpha} < w$. If $w (\alpha) \in \hat{\Phi}^+$ then $w < w r_{\alpha}$.
\end{proposition}

\begin{proof}
  Suppose that $w (\alpha) \in \hat{\Phi}^-$. Write $w = s_{i_1} s_{i_2}
  \cdots s_{i_m}$ a reduced expression. Then by Proposition~\ref{rootlist}
  $\alpha = s_{i_m} \cdots s_{i_{k + 1}} (\alpha_{i_k})$ for some $k$. Then
  \[ w r_{\alpha} = s_{i_1} \cdots \widehat{s_{i_k}} \cdots s_{i_m} < w,
     \hspace{2em} r_{\alpha} = (s_{i_m} \cdots s_{i_{k + 1}}) s_{i_k} (s_{i_{k
     + 1}} \cdots s_{i_m}) . \]
  where the caret denotes the omitted factor. This proves the first case.
  
  In the second case, $w (\alpha) \in \hat{\Phi}^+$ implies $w_0 w (\alpha)
  \in \hat{\Phi}^-$ so the first case is applicable and implies that $w_0 w
  r_{\alpha} < w_0 w$. Now $w_0 x < w_0 y$ is equivalent to $y < x$ and so $w
  < w r_{\alpha}$.
\end{proof}

\section{Upper triangularity of $m (u, v)$}

The Iwahori subgroup $J$ admits the {\tmem{Iwahori factorization}}
\[ J = T (\mathfrak{o}) \, N_- (\mathfrak{p}) \, N (\mathfrak{o}) . \]
The factors may be written in any order. This is a special case of the
following.

\begin{proposition}
  \label{iwahorifactorizationgen}If $x \in W$ then
  \[ x J x^{- 1} = T (\mathfrak{o}) (x J x^{- 1} \cap N) (x J x^{- 1} \cap
     N_-) . \]
\end{proposition}

\begin{proof}
  It follows from Matsumoto {\cite{Matsumoto}}, \ Lemme~5.4.2 on page 154 that
  \[ J = T (\mathfrak{o}) (J \cap w N w^{- 1}) (J \cap w N_- w^{- 1}) . \]
  Taking $w = x^{- 1}$ and conjugating gives the result.
\end{proof}

\begin{proposition}
  \label{yncriterionone}If $b \in B$ and $x, y \in W$ and if $y b \in B x J$
  then $x \leqslant y$.
\end{proposition}

\begin{proof}
  Using the Iwahori factorization of $J$ we may write $y b = b'' x n_- b'$
  where $b'' \in B$, $n_- \in N_- (\mathfrak{p})$, and $b' \in B
  (\mathfrak{o})$. Then $y b (b')^{- 1} = b'' x n_- \in B x B_-$ where $B_-$
  is the opposite Borel subgroup to $B$, so $B y B \cap B x B_- \neq
  \varnothing$. By Corollary 1.2 in Deodhar~{\cite{DeodharGeometric}} it
  follows that $x \leqslant y$.
\end{proof}

\begin{proposition}
  \label{yncriterionthree}Suppose that $n = n_1 n_2$ with $n_1, n_2 \in N$,
  and that $x n \in B x J$, $x n_1 x^{- 1} \in N$ and $x n_2 x^{- 1} \in N_-$.
  Then $n_2 \in N (\mathfrak{o})$.
\end{proposition}

\begin{proof}
  We write $x n = b x k$ with $k \in J$, so $x n_1 x^{- 1} \cdot x n_2 x^{- 1}
  = b x k x^{- 1}$. Then by Proposition~\ref{iwahorifactorizationgen} we write
  $x k x^{- 1} = a n_+ n_-$ with $a \in T (\mathfrak{o})$, $n_+ \in x J x^{-
  1} \cap N$ and $n_- \in x J x^{- 1} \cap N_-$. So
  \[ x n_1^{- 1} x^{- 1} b a n_+ = x n_2 x^{- 1} n_-^{- 1} . \]
  Here the left-hand side is in $B$ and the right hand side is in $N_-$, so
  both sides are $1$. Thus $n_2 = x^{- 1} n_- x \in N (\mathfrak{o})$.
\end{proof}

\begin{proposition}
  \label{yncriteriontwo}If $n \in N$ and $x \in W$, and $x n x^{- 1} \in N_-$,
  and if $x n \in B x J$ then $n \in N (\mathfrak{o})$.
\end{proposition}

\begin{proof}
  This is the special case of the previous Proposition with $n_1 = 1$.
\end{proof}

\begin{theorem}
  \label{uppertriangular}If $(M_v \psi_u) (1) \neq 0$ then $u \leqslant v$.
  Moreover $(M_u \psi_u) (1) = 1$.
\end{theorem}

\begin{proof}
  We may write
  \[ (M_v \psi_u) (1) = \int_{N \cap v N_- v^{- 1}} \psi_u (v^{- 1} n) \, d n.
  \]
  If this is nonzero, then $\psi_u (v^{- 1} n) \neq 0$ for some $n \in N$.
  Find $w \in W$ such that $v^{- 1} n \in B w^{- 1} J$. Then by definition of
  $\psi_u$ we have $w \geqslant u$. By Proposition~\ref{yncriterionone} $w^{-
  1} \leqslant v^{- 1}$ or $w \leqslant v$ and therefore $u \leqslant v$.
  
  Now if $u = v$ then
  \[ (M_u \psi_u) (1) = \int_{N \cap u N_- u^{- 1}} \psi_u (u^{- 1} n) \, d n.
  \]
  If $\psi_u (u^{- 1} n) \neq 0$ then by definition of $\psi_u$ we have $u^{-
  1} n \in B w^{- 1} J$ for some $w$ such that $w \geqslant u$. By
  Proposition~\ref{yncriterionone}, $w \leqslant u$ and so $w = u$. Now by
  Proposition~\ref{yncriteriontwo} $n \in N (\mathfrak{o})$. Thus the domain
  of integration can be taken to be $N (\mathfrak{o}) \cap w N_-
  (\mathfrak{o}) w$. On this domain, the integrand is $1$, and the measure is
  normalized so that the volume of $N (\mathfrak{o}) \cap w N_- (\mathfrak{o})
  w$ is $1$. Hence $(M_u \psi_u) (1) = 1$.
\end{proof}

\begin{proposition}
  \label{intertwiningofphi1}If $s = s_{\alpha}$ is a simple reflection then
  \[ M_s \left( {^{\chi} \phi_1} \right) = \frac{1}{q} \left( {^{^s \chi}
     \phi_s} \right) + \left( 1 - \frac{1}{q} \right)
     \frac{\tmmathbf{z}^{\alpha}}{1 -\tmmathbf{z}^{\alpha}} \left( {^{^s \chi}
     \phi_1} \right) . \]
\end{proposition}

\begin{proof}
  See Casselman~{\cite{CasselmanSpherical}}, Theorem~3.4.
\end{proof}

\section{Hecke algebra}

It is shown by Rogawski~{\cite{Rogawski}} that one may use the Iwahori Hecke
algebra to express the intertwining operators. We will review this method. See
also Reeder~{\cite{Reeder}} and Haines, Kottwitz and
Prasad~{\cite{HainesKottwitzPrasad}}.

We assume that the split semisimple group $G$ is simply-connected. There is no
loss of generality in assuming this for the purpose of computing the
intertwining operators and Casselman basis.

There are two Weyl groups which we must consider. There is the affine Weyl
group $W_{\tmop{aff}}$ which is $N_G (T (F)) / T (\mathfrak{o})$, and the
ordinary Weyl group $N_G (T (F)) / T (F)$. Following Iwahori and
Matsumoto~{\cite{IwahoriMatsumoto}}, {\cite{Iwahori}}, these Weyl groups and
their Hecke algebras may be described as follows. Let $\sigma_i =
s_{\alpha_i}$ be the simple reflections, where $\alpha_i$ are the simple roots
in $\hat{\Phi}$. Then $\sigma_i$ and $\sigma_j$ commute unless $i$ and $j$ are
adjacent nodes in the Dynkin diagram, in which case they satisfy the braid
relation; assuming $G$ is simply-laced, this has the form $\sigma_i \sigma_j
\sigma_i = \sigma_j \sigma_i \sigma_j$. Then $\sigma_1, \cdots, \sigma_r$
generate $W$. Another generator $\sigma_0$ is needed for $W_{\tmop{aff}}$.
Since we are assuming that $G$ is simply-connected, then $\sigma_0, \cdots,
\sigma_r$ generate $W_{\tmop{aff}}$ with generators and relations as above
except that one uses the extended Dynkin diagram to decide whether $i$ and $j$
are adjacent.

The Iwahori Hecke algebra is the convolution ring of compactly supported
functions $f$ on $G$ such that $f (k g k') = f (g)$ when $k, k' \in J$. Its
structure was determined by Iwahori and Matsumoto~{\cite{IwahoriMatsumoto}}.
Normalizing the Haar measure so that $J$ has volume $1$, let $t_w$ be the
characteristic function of $J w J$, and if $1 \leqslant i \leqslant r$ let
$t_i$ denote $t_{\sigma_i}$. The $t_w$ with $w \in W_{\tmop{aff}}$ form a
basis, and the $t_i$ form a set of algebra generators The $t_i$ satisfy the
same braid relations as the $s_i$, but the relation $\sigma_i^2 = 1$ is
replaced by $t_i^2 = (q - 1) t_i + q$.

The subalgebra elements of $H_{\tmop{aff}}$ consisting of functions that are
supported in $K$ is the finite Iwahori Hecke algebra $H$. Thus $\dim (H) =
|W|$ but $H_{\tmop{aff}}$ is infinite-dimensional The subalgebra $H$ has
generators $t_1, \cdots, t_r$ but omits $t_0$.

With notation as in the introduction, $V (\chi)^J$ is a module for
$H_{\tmop{aff}}$. If $\phi \in H$ and $f \in V (\chi)$ then $\phi f (g) =
\int_G \phi (h) f (g h) \, d h$.

We define a vector space isomorphism $\alpha = \alpha (\chi) : V (\chi)^J
\longrightarrow H$ as follows. If $F \in V (\chi)^J$ then let $\alpha (F) = f$
where $f$ is the function $f (g) = F (g^{- 1})$ if $g \in K$, 0 if $g \nin K$.
It may be checked using the Iwahori factorization that $\alpha (F) \in H$. Now
$V (\chi)^J$ is a left-module for $H$ (since $H \in H_{\tmop{aff}})$ and so is
$H$. It is easy to check that $\alpha$ is a homomorphism of left $H$-modules.
Now let $w \in W$ and define a map $\mathcal{M}_w =\mathcal{M}_{w,
\tmmathbf{z}} : H \longrightarrow H$ by requiring the diagram:
 \[\begin{CD}
      V (\chi)^J @>M_w >> V (^w \chi)^J\\
      @VV \small{\alpha (\chi)} V @VV \alpha(^w \chi) V\\
      H  @>\mathcal{M}_w >> H
 \end{CD}\]
to be commutative. If $w \in W$ let us define $\mu_{\tmmathbf{z}} (w)
=\mathcal{M}_w (1_H) \in H$, where $1_H$ is the unit element in the ring $H$.
Note that $\alpha_{\chi} (\phi_1) = 1_H$, so
\[ \mu_{\tmmathbf{z}} (w) = \alpha_{} (^w \chi) \phi_1 . \]
\begin{proposition}
  We have
  \[ \mathcal{M}_w (h) = h \cdot \mu_{\tmmathbf{z}} (w) \]
  for all $h \in H$.
\end{proposition}

\begin{proof}
  $\mathcal{M}_w$ is a homomorphism of left $H$-modules, where $H$, being a
  ring, is a bimodule. \ Therefore $ \mathcal{M}_w (h) =\mathcal{M}_w (h \cdot
  1) = h\mathcal{M}_w (1) = h \cdot \mu_{\tmmathbf{z}} (w)$.
\end{proof}

\begin{lemma}
  If $l (w_1 w_2) = l (w_1) + l (w_2)$ then
  \[ \mu_{\tmmathbf{z}} (w_1 w_2) = \mu_{\tmmathbf{z}} (w_2) \mu_{w_2
     \tmmathbf{z}} (w_1) . \]
\end{lemma}

\begin{proof}
  We have $M_{w_1 w_2} = M_{w_1} \circ M_{w_2}$. Therefore this follows from
  the commutativity of the diagram:
\[\begin{CD}
V (\chi)^J @> M_{w_2}>> V (^{w_2} \chi)^J @> M_{w_1}>> V (^{w_1 w_2} \chi)^J\\
@VV \alpha(\chi) V @VV \alpha(^{w_2} \chi) V @VV \alpha (^{w_1 w_2} \chi) V\\
H @>> \mathcal{M}_{w_2} > H @>> \mathcal{M}_{w_1} > H
\end{CD}
\]

\end{proof}

\begin{lemma}
  \label{srintlemma}If $w = \sigma_i$ is a simple reflection, then
  $\mathcal{M}_w (1) = \frac{1}{q} t_i + (1 - \frac{1}{q})
  \frac{\tmmathbf{z}^{\alpha_i}}{1 -\tmmathbf{z}^{\alpha_i}} .$
\end{lemma}

\begin{proof}
  This follows from Proposition~\ref{intertwiningofphi1}.
\end{proof}

We will denote $\alpha_{\chi} (\psi_u) = \psi (u)$. Note that this element of
$H$ is independent of $\chi$: it is just the union of the characteristic
functions of the double cosets $J w J$ with $w \geqslant u$. If $f \in H$, let
$\Lambda (f)$ denote the coefficient of $1$ in the expansion of $f$ in terms
of the basis elements. Then
\[ m_{\tmmathbf{z}} (u, v) = \Lambda (\psi (u) \mu_{\tmmathbf{z}} (v)) . \]

Let us introduce the following notation. If $f, g \in H$ and $x \in W$, we
will write $f \equiv g$ mod $x$ if the only $t_w$ ($w \in W$) that have
nonzero coefficient in $f - g$ are those with $w \geqslant x$.

\begin{proposition}
  \label{humpfgen}Let $x, y \in W$ let $s$ be a simple reflection. Assume $x
  \leqslant y$.
  
  (a) Either $x s \leqslant y$ or $x s \leqslant y s$.
  
  (b) Either $x \leqslant y s$ or $x s \leqslant y s$.
\end{proposition}

\begin{proof}
  Part (a) is proved in Humpfreys, {\tmem{Reflection Groups and Coxeter
  Groups}}, Proposition~5.9. For (b), since $W$ is a finite Weyl group, it has
  a long element $w_0$ and $w_0 y < w_0 x$. Therefore by (a) either $w_0 y s <
  w_0 x$ or $w_0 y s < w_0 x y$, which implies (b).
\end{proof}

\begin{proposition}
  \label{stableinterval}Let $u \in W$ and let $s$ be a simple reflection.
  
  (a) Assume that $u s > u$. Then for all $x \in W$ we have $x \geqslant u$
  if and only if $x s \geqslant u$.
  
  (b) Assume that $u s < u$. Then for all $x \in W$ we have $x \leqslant u$ if
  and only if $x s \leqslant u$.
\end{proposition}

\begin{proof}
  For (a), if $x \geqslant u$ then either $x s \geqslant u$ or $x s \geqslant
  u s$ by Proposition~\ref{humpfgen}, but if $u s > u$, both cases imply $x s
  \geqslant u$. Conversely if $x s \geqslant u$ the same argument shows $x
  \geqslant u$. This proves (a) and (b) is similar.
\end{proof}

If $s$ is a simple reflection, let us denote
\begin{equation}
  \label{ominusdef} u \ominus s = \left\{ \begin{array}{ll}
    u & \text{if $u < u s$,}\\
    u s & \text{if $u s < u$} .
  \end{array} \right.
\end{equation}
\begin{proposition}
  \label{ominusprop}$\text{If $x \geqslant u$ and $y \leqslant s$ then $x y
  \geqslant u \ominus s$} .$
\end{proposition}

\begin{proof}
  This follows from Proposition~\ref{stableinterval}.
\end{proof}

\begin{proposition}
  \label{firstusefulfact}Let $s$ be a simple reflection, and let $u \in W$
  such that $u s > u$. Then $\psi (u) t_s = q \psi (u)$ and $\psi (u)
  \mu_{\tmmathbf{z}} (s) = \left( \frac{1 - q^{- 1} \tmmathbf{z}^{\alpha}}{1
  -\tmmathbf{z}^{\alpha}} \right) \psi (u)$.
\end{proposition}

\begin{proof}
  The second conclusion follows from the first and Lemma~\ref{srintlemma}, so
  we prove $\psi (u) t_s = q \psi (u)$. By Proposition~\ref{stableinterval}
  $\{x \in W| x \geqslant u\}$ is stable under right multiplication by $s$, so
  we may write $\psi (u)$ as a sum of terms of the form $t_x + t_{x s}$ with
  $x s > x$. But
  \[ (t_x + t_{x s}) (t_s - q) = t_x (1 + t_s) (t_s - q) = 0 \]
  so $\psi (u) (t_s - q) = 0$.
\end{proof}

\begin{proposition}
  \label{secondusefulfact}Let $s$ be a simple reflection, and let $u \in W$
  such that $u s > u$. Then
  \begin{equation}
    \label{psiustsid} \text{$\psi (u s) t_s \equiv q \psi (u)$ mod $u s$},
    \hspace{2em} \text{$\psi (u s) \mu_{\tmmathbf{z}} (s) \equiv \psi (u)$ mod
    $u s$.}
  \end{equation}
\end{proposition}

\begin{proof}
  The first equation in (\ref{psiustsid}) implies the second since by
  Lemma~\ref{srintlemma} $\mu_{\tmmathbf{z}} (s)$ differs from $\frac{1}{q}
  t_s$ by a scalar and $\psi (u s) \equiv 0$ mod $u s$. We prove the first
  equation.
  
  Let us determine the coefficient of $t_x$ in $\psi (u s) t_s$ under the
  assumption that $x \ngeqslant u s$. We will show that this coefficient
  equals $q$ if $x \geqslant u$ and $0$ otherwise. This will prove the
  Proposition since this is also the coefficient of $t_x$ in $q \psi (u)$.
  
  If $x \geqslant u$ then by Proposition~\ref{humpfgen} either $u s \leqslant
  x$ or $u s \leqslant x s$. Since we are assuming that $x \ngeqslant u s$ it
  follows that $u s \leqslant x s$. Hence $\psi (u s)$ has a term $t_{x s}$
  but no term $t_x$. Therefore the only term in the sum
  \begin{equation}
    \label{psisumexpand} \psi (u s) t_s = \sum_{z \geqslant u s} t_z t_s
  \end{equation}
  that can contribute to the coefficient of $t_x$ is the term is $t_{x s}
  t_s$. Since $x s \geqslant u s$ but $x \ngeqslant u s$ we have $x s > x$.
  Thus $t_{x s} = t_x t_s$ and $t_{x s} t_s = t_x q$. Therefore the
  coefficient of $t_x$ is $q$.
  
  If $x \ngeqslant u$ then we claim that there is no contribution to $t_x$
  from any term in the sum (\ref{psisumexpand}). Indeed, the only $z$ which
  could produce a contribution would $z = x$ or $z = x s$, but the condition
  $z \geqslant u s$ is not satisfied for these. Indeed, $x \ngeqslant u s$
  since $x \ngeqslant u$. If $x s \geqslant u s$ then by
  Proposition~\ref{humpfgen}, either $x \geqslant u s$ or $x \geqslant u$.
  Since $u s > u$ we have $x \geqslant u$ in either case, contradicting our
  assumption.
\end{proof}

\section{Proof of Theorem \ref{maintheorem}}

In this section we will not assume that $\hat{\Phi}$ is simply-laced.

\begin{theorem}
  Suppose that there exist reduced words
  \begin{eqnarray*}
    v & = & s_1 \ldots s_n\\
    u & = & s_1 \cdots \widehat{s_{i_1}} \cdots \widehat{s_{i_2}} \cdots
    \cdots \widehat{s_{i_m}} \cdots s_n,
  \end{eqnarray*}
  so that $l (v) = n$ and $l (u) = n - m$. Suppose moreover that $|S (u, v) |
  = l (v) - l (u)$ and that
  \[ \left. S (u, v) =\{s_n s_{n - 1} \cdots s_{i_k + 1} (\alpha_{i_k}) |1
     \leqslant k \leqslant m \right\} . \]
  Then
  \[ m (u, v) = \prod_{\alpha \in S (u, v)} \frac{1 - q^{- 1}
     \tmmathbf{z}^{\alpha}}{1 -\tmmathbf{z}^{\alpha}} . \]
\end{theorem}

\medbreak\noindent
The reflections $r = r_{\alpha}$ with $\alpha \in S (u, v)$ such that $u
\leqslant v r < v$ are precisely the $s_n s_{n - 1} \cdots s_{i_k + 1} s_{i_k}
s_{i_k + 1} \cdots s_n$ with $1 \leqslant k \leqslant m$, and so $u \leqslant
s_1 \cdots \widehat{s_{i_k}} \cdots s_n < v$. Therefore the hypothesis that $u
= s_1 \cdots \widehat{s_{i_1}} \cdots \widehat{s_{i_2}} \cdots \cdots
\widehat{s_{i_m}} \cdots s_n$ is equivalent to the assumption that $s_1 \cdots
s_n$ is a good word for $v$ with respect to $u$. It follows that this theorem
is equivalent to Theorem \ref{maintheorem}.

\medbreak\noindent
\begin{proof}
  Here $s_i$ is a simple reflection. Let $\alpha_i$ be the corresponding
  simple root. We will write $\mu (s_n) = \mu_{\tmmathbf{z}} (s_n)$, $\mu
  (s_{n - 1}) = \mu_{s_n (\tmmathbf{z})} (s_{n - 1})$, ... , suppressing the
  dependence on the spectral parameters. We have $m (u, v) = \Lambda (\psi (u)
  \mu_{\tmmathbf{z}} (v))$ where we may write $\psi (u) \mu_{\tmmathbf{z}}
  (v)$ as a sum of terms
  \begin{eqnarray*}
    [\psi (s_1 \scdots \widehat{s_{i_1}} \scdots \widehat{s_{i_m}} \scdots s_n)
    \mu (s_n) -_{} \psi (s_1 \scdots \widehat{s_{i_1}} \scdots \widehat{s_{i_m}}
    \scdots s_{n - 1})] \mu (s_{n - 1}) \scdots \mu (s_1) & + & \\
    {}[\psi (s_1 \scdots \widehat{s_{i_1}} \scdots \widehat{s_{i_m}} \scdots s_{n
    - 1}) \mu (s_{n - 1}) -_{} \psi (s_1 \scdots \widehat{s_{i_1}} \scdots
    \widehat{s_{i_m}} \scdots s_{n - 2})] \mu (s_{n - 2}) \scdots \mu (s_1) & +
    & \\
    \vdots &  & \\
    {}[\psi (s_1 \scdots \widehat{s_{i_1}} \scdots \widehat{s_{i_m}}) \mu
    (s_{i_m}) -_{} C (m) \psi (s_1 \scdots \widehat{s_{i_1}} \scdots
    \widehat{s_{i_m}})] \mu (s_{i_m - 1}) \scdots \mu (s_1) & + & \\
    C (m) [\psi (s_1 \scdots \widehat{s_{i_1}} \scdots s_{i_m - 1}) \mu (s_{i_m
    - 1}) -_{} \psi (s_1 \scdots \widehat{s_{i_1}} \scdots s_{i_m - 2})] \mu
    (s_{i_m - 2}) \scdots \mu (s_1) & + & \\
    \vdots &  & \\
    C (m) \cdots C (1) [\psi (s_1) \mu (s_1) - \psi (1)] & + & \\
    C (m) \cdots C (1) \psi (1), &  & 
  \end{eqnarray*}
  where
  \[ C (k) = \frac{1 - q^{- 1} \tmmathbf{z}^{\gamma_k}}{1
     -\tmmathbf{z}^{\gamma_k}}, \hspace{2em} \gamma_k = s_n s_{n - 1} \cdots
     s_{i_k + 1} (\alpha_{i_k}) . \]
  The summation telescopes with the terms cancelling in pairs. We will show
  that $\Lambda$ annihilates every term except the last, so that $m (u, v) = C
  (m) \cdots C (1)$, as required.
  
  We note that the terms of the form
  \[ \prod_{j > k} C (j) [\psi (s_1 \cdots \widehat{s_{i_1}} \cdots
     \widehat{s_{i_k}}) \mu (s_{i_k}) -_{} C (k) \psi (s_1 \cdots
     \widehat{s_{i_1}} \cdots \widehat{s_{i_k}})] \mu (s_{i_k + 1}) \cdots \mu
     (s_1) \]
  are equal to zero by Proposition~\ref{firstusefulfact}. The spectral
  parameter for $\mu (s_{i_k})$ is $s_{i_{k + 1}} \cdots s_{i_n}
  (\tmmathbf{z})$ so
  \[ C (k) = \frac{1 - q^{- 1} \tmmathbf{z}^{\gamma_k}}{1
     -\tmmathbf{z}^{\gamma_k}} = \frac{1 - q^{- 1} (s_{i_{k + 1}} \cdots
     s_{i_n} (\tmmathbf{z}))^{\alpha_{i_k}}}{1 - (s_{i_{k + 1}} \cdots s_{i_n}
     (\tmmathbf{z}))^{\alpha_{i_k}}}, \]
  as in Proposition~\ref{firstusefulfact}.
  
  Each remaining terms is a constant times
  \[ [\psi (s_1 \cdots \widehat{s_{i_1}} \cdots \widehat{s_{i_2}} \cdots s_j)
     \mu (s_j) - \psi (s_1 \cdots \widehat{s_{i_1}} \cdots \widehat{s_{i_2}}
     \cdots s_{j - 1})] \mu (s_{j - 1}) \cdots \mu (s_1) . \]
  By Proposition~\ref{firstusefulfact} we have
  \[ \psi (s_1 \cdots \widehat{s_{i_1}} \cdots \widehat{s_{i_2}} \cdots s_j)
     \mu (s_j) - \psi (s_1 \cdots \widehat{s_{i_1}} \cdots \widehat{s_{i_2}}
     \cdots s_{j - 1}) \geqslant s_1 \cdots \widehat{s_{i_1}} \cdots
     \widehat{s_{i_2}} \cdots s_j, \]
  and so by Proposition~\ref{ominusprop} this term is $\geqslant s_1 \cdots
  \widehat{s_{i_1}} \cdots \widehat{s_{i_2}} \cdots s_j \ominus s_{j - 1}
  \ominus s_{j - 2} \ominus \cdots \ominus s_1$ in the
  notation~(\ref{ominusdef}). Thus this term is annihilated by $\Lambda$
  unless

  \begin{equation}
    \label{ominusidentity} s_1 \cdots \widehat{s_{i_1}} \cdots
    \widehat{s_{i_2}} \cdots s_j \ominus s_{j - 1} \ominus s_{j - 2} \ominus
    \cdots \ominus s_1 = 1.
  \end{equation}
  We will assume this and deduce a contradiction. If this is true, then we may
  write
  \begin{equation}
    \label{usegmentrewrite} s_1 \cdots \widehat{s_{i_1}} \cdots
    \widehat{s_{i_2}} \cdots s_j = s_1 \cdots \widehat{s_{j_1}} \cdots
    \widehat{s_{j_k}} \cdots s_{j - 1},
  \end{equation}
  where $j_k, j_{k - 1}, \cdots$ $j_1$ are the locations in
  (\ref{ominusidentity}) where the $\ominus$ is {\tmem{not}} a descent. In
  other words, the left hand side of (\ref{ominusidentity}) is of the form
  $s_1 \cdots \widehat{s_{i_1}} \cdots \widehat{s_{i_2}} \cdots s_j s_{j - 1}
  \cdots \widehat{s_{j_k}} \cdots \widehat{s_{j_1}} \cdots s_1$, and we have
  moved terms to the other side to obtain (\ref{usegmentrewrite}).
  
  Now using (\ref{usegmentrewrite}) we may write
  \begin{equation}
    \label{urewrite} u = s_1 \cdots \widehat{s_{i_1}} \cdots s_n = s_1 \cdots
    \widehat{s_{j_1}} \cdots \widehat{s_{j_k}} \cdots s_{j - 1} \widehat{s_j}
    s_{j + 1} \cdots \widehat{s_{i_m}} \cdots s_n,
  \end{equation}
  for we have substituted the right-hand side of (\ref{usegmentrewrite}) for
  an initial segment in the word representing $u$. Now let $\delta = s_n s_{n
  - 1} \cdots s_{j + 1} (\alpha_j)$. Then
  \[ v \, r_{\delta} = s_1 \cdots \widehat{s_j} \cdots s_n, \]
  with only one omitted entry $s_j$. Clearly $v r_{\delta} < v$ and by
  (\ref{urewrite}) we have $u \leqslant v r_{\delta}$. Thus $\delta \in S (u,
  v)$. This is a contradiction, however, because the list
  \[ s_n \cdots s_{k + 1} (\alpha_k) \]
  of positive roots $\alpha$ such that $v (\alpha) \in \hat{\Phi}^-$ has no
  repetitions by Proposition~\ref{rootlist}. But $j$ is not among the set
  $\{j_1, \cdots, j_m \}$.
\end{proof}

\section{Towards Conjecture~\ref{conjecture2}}

\begin{proposition}
  \label{conjecture3dual}If Conjecture~\ref{conjecture3} is true and if $u <
  v$ and $P_{w_0 v, w_0 u} = 1$ then there exists $\beta \in \hat{\Phi}^+$
  such that $u \leqslant t \leqslant v$ if and only if $u \leqslant r_{\beta}
  t \leqslant v$.
\end{proposition}

\begin{proof}
  The conjecture implies that there exists $\gamma$ such that $w_0 v \leqslant
  t \leqslant w_0 u$ if and only if $w_0 v \leqslant r_{\gamma} t \leqslant
  w_0 u$. Then we may take $\beta = - w_0 (\gamma)$ so that $r_{\beta} = w_0
  r_{\gamma} w_0$ and
  \begin{eqnarray*}
    u \leqslant t \leqslant v \hspace{1em} \Leftrightarrow \hspace{1em} w_0 v
    \leqslant w_0 t \leqslant w_0 u \hspace{1em} \hspace{1em} \Leftrightarrow
    &  & w_0 v \leqslant r_{\gamma} w_0 t \leqslant w_0 u\\
    \hspace{1em} \Leftrightarrow &  & u \leqslant w_0 r_{\gamma} w_0 t
    \leqslant v.
  \end{eqnarray*}
\end{proof}

Although we do not see how to deduce Conjecture~\ref{conjecture2} from
Conjecture~\ref{conjecture1}, we have the following special case.

\begin{theorem}
  \label{conjtwoprogress}Conjectures~\ref{conjecture1} and~\ref{conjecture3}
  imply that if $P_{u, v} = P_{w_0 v, w_0 u} = 1$ then~(\ref{mtildeval}) is
  satisfied.
\end{theorem}

\begin{proof}
  With notation as in Conjecture \ref{conjecture3} we note that the map $t
  \longmapsto t' = r_{\beta} t$ is a bijection of the set $\{t| u \leqslant t
  \leqslant v\}$ to itself such that $l (t) \nequiv l (t')$ mod 2 and such
  that
  \begin{equation}
    \label{equalityofreflections} \{\alpha |u \leqslant t r_{\alpha} \leqslant
    v\}=\{\alpha |u \leqslant t' r_{\alpha} \leqslant v\}.
  \end{equation}
  Indeed the property that $r_{\beta}$ has, applied to $t r_{\alpha}$ instead
  of $t$, implies that $u \leqslant t r_{\alpha} \leqslant v$ if and only if
  $u \leqslant r_{\beta} t r_{\alpha} \leqslant v$.
  
  Let $M$ and $\tilde{M}$ be the matrices with coefficients $m (u, v)$ and
  $\tilde{m} (u, v)$. We know that these matrices are upper triangular with
  respect to the Bruhat order, that is, $m (u, v) = \tilde{m} (u, v) = 0$
  unless $u \leqslant v$. Assuming Conjecture 1 if $P_{w_0 v, w_0 u} = 1$ then
  $m (u, v) = m' (u, v)$ where
  \begin{equation}
    \label{misprod} m' (u, v) = \prod_{\tmscript{\begin{array}{c}
      \alpha \in \hat{\Phi}^+\\
      u \leqslant v r_{\alpha} < v
    \end{array}}} R (\alpha), \hspace{2em} R (\alpha) = \frac{1 - q^{- 1}
    \tmmathbf{z}^{\alpha}}{1 -\tmmathbf{z}^{\alpha}} .
  \end{equation}
  According to Conjecture 2 if $P_{u, v} = 1$ then we should have $\tilde{m}
  (u, v) = \tilde{m}' (u, v)$ where

  \begin{equation}
    \label{mtilprod} \tilde{m}' (u, v) = (- 1)^{l (v) - l (u)}
    \prod_{\tmscript{\begin{array}{c}
      \alpha \in \hat{\Phi}^+\\
      u \leqslant u r_{\alpha} < v
    \end{array}}} R (\alpha) .
  \end{equation}
  By induction, we may assume that the counterexample minimizes $l (v) - l
  (u)$. If $u < t \leqslant v$ then $P_{t, v} = 1$ and $P_{w_0 t, w_0 u} = 1$.
  Therefore $m (u, t) = m' (u, t)$ and $\tilde{m} (t, v) = \tilde{m}' (t, v)$.
  Now since $M$ and $\tilde{M}$ are inverse matrices, we have
  \[ \sum_{u \leqslant t \leqslant v} m (u, t) \tilde{m} (t, v) = 0, \]
  and in this relation $m (u, t) = m' (u, t)$ is assumed for all $t$ and
  $\tilde{m} (t, v) = \tilde{m}' (t, v)$ is proved for all $t$ except when $t
  = u$. Therefore $\tilde{m} (u, v) = \tilde{m}' (u, v)$ will follow if we
  prove
  \[ \sum_{u \leqslant t \leqslant v} m' (u, t) \tilde{m}' (t, v) = 0. \]
  We have
  \[ m' (u, t) \tilde{m}' (t, v) = (- 1)^{l (v) - l (t)} \prod_{u \leqslant t
     r_{\alpha} \leqslant t} R (\alpha) \prod_{t \leqslant t r_{\alpha}
     \leqslant v} R (\alpha) = (- 1)^{l (v) - l (t)} \prod_{u \leqslant t
     r_{\alpha} \leqslant v} R (\alpha) \]
  because by Proposition~\ref{wralphacrit} we always have either $t r_{\alpha}
  < t$ or $t < t r_{\alpha}$. Using Conjecture~\ref{conjecture3} and
  Proposition~\ref{conjecture3dual} there is a bijection $t \mapsto r_{\beta}
  t$ for some reflection $r_{\beta}$ that stabilizes the Bruhat interval $u
  \leqslant t \leqslant v$. With $u \leqslant t \leqslant v$ this means that
  $u \leqslant t r_{\alpha} \leqslant v$ if and only if $u \leqslant r_{\beta}
  t r_{\alpha} \leqslant v$ and
  \[ (- 1)^{l (v) - l (t)} \prod_{u \leqslant t r_{\alpha} \leqslant v} R
     (\alpha) = - (- 1)^{l (v) - l (r_{\beta} t)} \prod_{u \leqslant r_{\beta}
     t r_{\alpha} \leqslant v} R (\alpha), \]
  so the terms corresponding to $t$ and $r_{\beta} t$ cancel.
\end{proof}

\end{document}